\documentclass[11pt]{amsart}

\usepackage[utf8]{inputenc}
\usepackage[OT2,T1]{fontenc}
\usepackage[T1]{fontenc}
\usepackage[english]{babel}
\usepackage{lmodern}
\usepackage{textcomp}

\usepackage{amssymb, amsmath, amsthm}
\usepackage{float}
\usepackage{fullpage}
\usepackage{graphicx}
\usepackage[all]{xy}

\usepackage{hyperref}
\usepackage[usenames,dvipsnames]{color}
\usepackage{enumerate,enumitem}

\theoremstyle{plain}
\newtheorem{lemma}{Lemma}
\newtheorem{theorem}[lemma]{Theorem}
\newtheorem{corollary}[lemma]{Corollary}
\newtheorem{proposition}[lemma]{Proposition}

\theoremstyle{definition}

\newtheorem{example}[lemma]{Example}

\theoremstyle{remark}
\newtheorem{remark}[lemma]{Remark}

\newcommand{\p}{\mathfrak{p}}

\DeclareMathOperator{\M}{\mathcal{M}}

\DeclareMathOperator{\Ocal}{\mathcal{O}}
\DeclareMathOperator{\Epsilon}{\mathcal{E}}
\DeclareMathOperator{\Div}{\mathrm{Div}}
\DeclareMathOperator{\supp}{\mathrm{Supp}}
\DeclareMathOperator{\ord}{\mathrm{ord}}
\DeclareMathOperator{\End}{\mathrm{End}}

\renewcommand{\char}{\text{char}}

\def\phi{\varphi}

\usepackage{multirow}

\hypersetup{
	pdfauthor   = {Slob},
	pdftitle    = {Primitive divisors of sequences associated to elliptic curves over function fields},
	pdfsubject  = {},
	pdfkeywords = {},
	backref=true, pagebackref=true, hyperindex=true, colorlinks=true,
	breaklinks=true, urlcolor=blue, linkcolor=blue, citecolor=blue,
	bookmarks=true, bookmarksopen=true}

\begin{document}	
	\title{Primitive divisors of sequences associated to elliptic curves over function fields}
	%\date{\today}
	\begin{abstract}
		We study the existence of a Zsigmondy bound for a sequence of divisors associated to points on an elliptic curve over a function field. More precisely, let $k$ be an algebraically closed field, let $\mathcal{C}$ be a nonsingular projective curve over $k$, and let $K$ denote the function field of $\mathcal{C}$. Suppose $E$ is an ordinary elliptic curve over $K$ and suppose there does not exist an elliptic curve $E_0$ defined over $k$ that is isomorphic to $E$ over $K$. Suppose $P\in E(K)$ is a non-torsion point and $Q\in E(K)$ is a torsion point of order $r$. The sequence of points $\{nP+Q\}\subset E(K)$ induces a sequence of effective divisors $\{D_{nP+Q}\}$ on $\mathcal{C}$. We provide conditions on $r$ and the characteristic of $k$ for there to exist a bound $N$ such that $D_{nP+Q}$ has a primitive divisor for all $n\geq N$. This extends the analogous result of Verzobio in the case where $K$ is a number field.
	\end{abstract}
	\author[Slob]{Robert Slob}
	\address{%
		Robert Slob,
		Institut für Reine Mathematik,
		Universität Ulm,
		Helmholtzstrasse 18,
		89081 Ulm,
		Germany
	}
	\email{robert.slob@uni-ulm.de}
	\maketitle
	
	\section{Introduction}
	Let $K$ be a number field with ring of integers $\Ocal_K$. Let $E/K$ be an elliptic curve that is given by a Weierstrass equation with integral coefficients, and suppose $P\in E(K)$ is a non-torsion point. For each positive integer $n$, we can write $(x(nP))=\frac{A_n}{D_n^2}$, where $A_n$ and $D_n$ are coprime ideals in $\Ocal_K$. The sequence of ideals $\{D_n\}$ forms a \textit{divisibility sequence}, meaning that if $m$ and $n$ are positive integers with $m$ dividing $n$, then $D_m$ divides $D_n$. 
	
	Some famous sequences such as the Mersenne sequence and Lucas sequence are examples of divisibility sequences. The divisibility sequence obtained from a non-torsion point on an elliptic curve is an example of an \textit{elliptic divisibility sequence}, which were first studied by Morgan Ward \cite{ward1948memoir}. The book \cite[Chapter 10]{everest2003recurrence} of Everest et al. gives a gentle introduction into the subject of elliptic divisibility sequences and provides a great historical account. For an interesting connection between matrix divisibility sequences and (elliptic) divisibility sequences, see \cite{cornelissen2012matrix}. Additionally, see the introduction of [op. cit.] for some recent research and applications of (elliptic) divisibility sequences. 
	
	Returning to our sequence $\{D_n\}$, let $n$ be a positive integer, then we say that $D_n$ has a \textit{primitive divisor} if there exists a prime ideal $\p$ of $\Ocal_K$ that divides $D_n$ and does not divide $D_m$ for any $1\leq m<n$. If $K=\mathbb{Q}$, then $D_n$ is simply an integer, and in this case, it was proved by Silverman in 1988 that there exists a bound $N$ such that $D_n$ has a primitive divisor for all $n\geq N$ \cite{silverman1988wieferich}. Such a bound is sometimes called a \textit{Zsigmondy bound} in the literature, dating back to Zsigmondy's study of the divisibility sequence $d_n=a^n-b^n$ for $a>b>0$ positive coprime integers in the late 19th century. Zsigmondy showed that if $n\notin\{1,2,6\}$, then $d_n$ has a primitive divisor \cite{zsigmondy1892theorie}. This generalises an earlier result of Bang with $b$ equal to $1$, see \cite{bang1886taltheoretiske}. An immediate application of the existence of a Zsigmondy bound would be to try and use this result to search for large prime numbers. For this to be computationally feasible, one wants the values $D_n$ to be prime themselves. In this direction, the Chudnovsky brothers found some promising results in 1986 in their experiments for certain values of $D_n$ coming from elliptic divisibility sequences as above \cite{chudnovsky1986sequences}. However, later research indicated that these sequences may not be very suitable for this application \cite{einsiedler2001primes, everest2004primes}. Nevertheless, there are other applications. Elliptic nets are a generalisation of elliptic divisibility sequences, which have been used by Stange for applications in cryptography \cite{stange2007tate}. Additionally, there have been applications to a generalisation of Hilbert's tenth problem for large subrings of the rational numbers \cite{cornelissen2007elliptic, eisentrager2009descent, poonen2003hilbert}.
	
	A natural question is whether it is possible to extend Silverman's result to other fields. In 1999, Cheon and Hahn proved the result when $K$ is a number field \cite{cheon1999orders}. The fact that the sequence $\{D_n\}$ is a divisibility sequence plays a major role in both this and Silverman's proof. Effective versions of these theorems have been proved as well \cite{ingram2012uniform, verzobio2020some}. In a different direction, one can also consider other sequences of points in $E(K)$ and raise similar questions. Suppose $Q\in E(K)$ is a point with $Q\neq -nP$ for any positive integer $n$. For each positive integer $n$, we can then similarly write $(x(nP+Q))=\frac{A'_n}{{D'_n}^2}$ with $A'_n$ and $D'_n$ ideals in $\Ocal_K$ that are relatively prime. In general, the sequence of ideals $\{D'_n\}$ will no longer be a divisibility sequence, but one can still pose the question whether there exists a bound $N$ such that $D'_n$ has a primitive divisor for all $n\geq N$. For a number field as base field, questions related to this are considered in \cite{everest2005prime}, and Verzobio proves in \cite{verzobio2020primitive} that for $Q$ a torsion point of prime order $r$, such a bound exists. In a later note, Verzobio extended this result to the case where $Q$ is an arbitrary torsion point \cite{verzobio2020primitiveCM}. Actually, much more is proved in [op. cit.]. Namely, for $R\subset \End(E)$ a Dedekind domain, the author proves results concerning primitive divisors for the sequence of points $\{\alpha(P)+Q\}_{\alpha\in R}$, including a result when $Q$ is not assumed to be a torsion point. This is an extension of the work by Streng in \cite{streng2008divisibility}, where for $R\subset \End(E)$ an arbitrary subring, results concerning primitive divisors for the sequence of points $\{\alpha(P)\}_{\alpha\in R}$ are proved.
	
	In this paper, we extend one of the aforementioned results of Verzobio to the setting where $K$ is the function field of a nonsingular projective curve $\mathcal{C}$ over an algebraically closed field $k$ of characteristic $p$. Suppose $E/K$ is an elliptic curve with point at infinity $O\in E(K)$. We next state some results concerning elliptic surfaces, see for example \cite[Chapters III \& IV]{silverman2013advanced} for details. We can associate an elliptic surface to $E$, and among those there exists a minimal proper regular model, unique up to $K$-isomorphism. Fix such a minimal proper regular model for $E$ and denote it by $\Epsilon$. Suppose $R\in E(K)$ is a point, then we obtain an associated section $\sigma_R:\mathcal{C}\to \Epsilon$. Let $\Ocal$ denote the image of $\sigma_O$. If $R$ is non-zero, it can be shown that $\sigma_R^*(\Ocal)$ is an effective divisor on $\mathcal{C}$. Given a non-zero point $R\in E(K)$, we denote $D_R:=\sigma_R^*(\Ocal)\in \Div(\mathcal{C})$. 
	
	Then, given a sequence of non-zero points $\{P_n\}\subset E(K)$, we obtain a sequence of effective divisors $\{D_{P_n}\}\subset \Div(\mathcal{C})$. Extending the earlier definitions, we say that a sequence of effective divisors $\{\mathcal{D}_n\}\subset \Div(\mathcal{C})$ is a \textit{divisibility sequence} if for all positive integers $m,n$ with $m$ dividing $n$, we have that $\mathcal{D}_n-\mathcal{D}_m$ is effective. Similarly, given a positive integer $n$, we say that $\mathcal{D}_n$ has a \textit{primitive divisor} if there exists $\gamma$ in the support of $\mathcal{D}_n$ such that $\gamma$ does not lie in the support of $\mathcal{D}_m$ for any $1\leq m<n$. We next state some results from \cite{ingram2012algebraic} and \cite{naskrkecki2016divisibility}, where the former concerns $\char(k)=p=0$ and the latter $p>0$. Suppose $P\in E(K)$ is a non-torsion point, then the sequence of divisors $\{D_{nP}\}$ is a divisibility sequence. Suppose that $E$ is ordinary and that $E$ is not isomorphic over $K$ to some elliptic curve $E_0/k$. Additionally, suppose $p\neq 2,3$, then there exists a bound $N$ such that for all $n\geq N$, $D_{nP}$ has a primitive divisor. Given these results, it is natural to pose the question whether the aforementioned results of Verzobio over number fields also hold in the setting of $K$ a function field as above. In this paper, we study one of these results. That is, we study the following question: let $Q\in E(K)$ be a torsion point of order $r$ and consider the sequence of divisors $\{D_{nP+Q}\}$, does there then exist a bound $N$ such that $D_{nP+Q}$ has a primitive divisor for all $n \geq N$? We prove that this is indeed true if we assume some minor conditions on $p$ and $r$. More precisely, we prove the following theorem.
	
	\begin{theorem}\label{thm:main}
		Let $k$ be an algebraically closed field of characteristic $p$, let $\mathcal{C}$ be a nonsingular projective curve over $k$ and let $K$ be the function field of $\mathcal{C}$. Suppose $E/K$ is an ordinary elliptic curve that is not isomorphic over $K$ to some elliptic curve $E_0/k$. Suppose $P\in E(K)$ is a non-torsion point and $Q\in E(K)$ is a torsion point of order $r$. If either $r=1$ and $p\neq2,3$ or the values of $p$ and $r$ are entries in Table \ref{tbl:mainthm}, then there exists a constant $N$ such that for all $n\geq N$, $D_{nP+Q}$ has a primitive divisor.
	\end{theorem}
	\begin{table}
	\begin{center}
		\begin{tabular}{|l||l|l|l|l|l|}
			\hline
			$p$ &$0$&$5$ &$7$&$11,13$&$\geq 17$ \\ \hline
			$r $  &$\geq2$   & $5$ or $\geq10$  &$\geq4$ &$\geq 3$ &$\geq 2$ \\\hline
	\end{tabular}\end{center}
	\caption{Pairs $(p,r)$ with $r>1$ for which $D_{nP+Q}$ has a primitive divisor for all $n$ sufficiently large.}
	\label{tbl:mainthm}
	\end{table}
\numberwithin{lemma}{section}
\setcounter{lemma}{0}
	\begin{remark}
	In an earlier version of this paper, we assumed in above theorem that $Q$ had prime order unequal to $p$. We required this assumption to prove the corresponding versions of Proposition \ref{prop:gammainsupprndaswell} and Corollary \ref{cor:boundbyassociatedd}. It was pointed out to the author that we could get around this assumption by Verzobio. Additionally, Ulmer has pointed out to the author that the paper \cite{ulmer2019transversality} of Ulmer and Urs{\'u}a could be used to lift this restriction in the $p=0$ case, see especially [Remark 2.4, op. cit.].
	\end{remark}
	This paper is organised as follows. In Section \ref{sec:prelims}, we recall some preliminaries on height functions and properties of the divisor associated to a point on an elliptic curve over a function field. Afterwards, we present the proof of Theorem \ref{thm:main} in Section \ref{sec:prfmain}. Lastly, we discuss the necessity of some of the assumptions of Theorem \ref{thm:main} in Section \ref{sec:conditions}. In particular, we provide counterexamples if $E$ is not ordinary and we discuss the case where $k$ is not algebraically closed.
	
	\subsection*{Notation.} Throughout Sections \ref{sec:prelims} and \ref{sec:prfmain} of this paper, we fix the following notation. For $k$ a field, a \textit{curve} over $k$ is a scheme $X$ over $k$ that is integral, separated, of finite type, and of dimension $1$. We let $k$ be an algebraically closed field of characteristic $p\neq 2,3$. We let $\mathcal{C}$ be a nonsingular projective curve over $k$ and we let $K$ be the function field of $\mathcal{C}$. We let $E/K$ be an elliptic curve with point at infinity $O\in E(K)$. We assume that $E(K)$ has non-zero rank and is given by a Weierstrass equation in short form. Additionally, we assume that $E$ is not isomorphic over $K$ to some elliptic curve $E_0/k$, and if $p>0$, we assume that $E$ is ordinary. We let $\Epsilon$ be an elliptic surface associated to $E$ that is a minimal proper regular model. We let $P\in E(K)$ be a non-torsion point and we let $Q\in E(K)$ be a torsion point of order $r$. For a non-zero point $R\in E(K)$ and $\sigma_R:\mathcal{C}\to\Epsilon$ the associated section, we denote $D_R:=\sigma_R^*(\Ocal)\in \Div(\mathcal{C})$, where $\Ocal$ equals the image of the section $\sigma_O$. In Section \ref{sec:conditions}, we will use above notation as well, but we will relax some of the assumptions, which will be indicated clearly. Lastly, we will frequently use the big $O$ and little $o$ notation. The subscripts in the $O$ indicate that the chosen constant depends on these subscripts, e.g. for $\alpha,\beta\in\mathbb{R}$, $\alpha=\beta+O_{E,P}(1)$ means that $|\alpha-\beta|\leq C$ for some constant $C$ depending on $E$ and $P$.
	
	\section{Preliminaries}\label{sec:prelims}
	We first provide a more explicit description of the divisor associated to a non-zero point in $E(K)$. 
	\begin{lemma}\label{lem:descriptiondiv}
		Suppose $R$ is a non-zero point in $E(K)$ and $\gamma\in \mathcal{C}(k)$. Let $E'/K$ be an elliptic curve given by a Weierstrass equation that is minimal at $\ord_{\gamma}$ and isomorphic to $E$ over $K$ via the isomorphism $\phi:E\to E'$, then
		\[\ord_{\gamma}D_R=\max\left\{0,-\frac{1}{2}\ord_{\gamma}(x(\phi(R)))\right\}.\]
	\end{lemma} 
	\begin{proof}[Proof]
		This is proved in \cite[Lemma 5.2]{ingram2012algebraic}, where we note that although the lemma stated there only concerns $D_{nP}$, the proof holds in this more general setting as well.
	\end{proof}
	\subsection{Heights}
	We next recall some properties of the (canonical) height map on $E$. We define the \textit{height} $h:E(K)\to \mathbb{Z}_{\geq 0}$ by 
	\[h(R)=\begin{cases}
		0,\ &\text{ if }R=O,\\
		\deg(x(R)), &\text{ otherwise}.\end{cases}\]
	The height of a non-zero point and the degree of its associated divisor are closely related. To show this, we require the following lemma.
	\begin{lemma}\label{lem:minimaleq}
		Let $\gamma\in \mathcal{C}(k)$, then there exists $u\in K^\times$ such that the change of coordinates $(x,y)\mapsto(u^2x,u^3y)$ is minimal at $\ord_{\gamma}$. 
	\end{lemma}
	\begin{proof}[Proof]
		By \cite[Proposition VII.1.3]{silverman2009arithmetic}, we know that there exists a change of variables with values in $K$ such that we obtain a minimal equation at $\ord_{\gamma}$ for $E$, say 
		\[y^2+a_1xy+a_3y=x^3+a_2x^2+a_4x+a_6.\]
		Let $R\subset K$ be the valuation ring corresponding to $\ord_{\gamma}$, then we also obtain from [loc. cit.] that a change of variables $(x,y)\mapsto (x+r,y+sx+t)$ with $r,s,t\in R$ again results in a minimal equation at $\ord_{\gamma}$. Since $\char(K)\neq 2$ and $a_1,a_3,-\frac{1}{2}\in R$, the change of coordinates $(x,y)\mapsto(x,y-\frac{1}{2}(a_1x+a_3))$ then results in a Weierstrass equation for $E$ that is minimal at $\ord_{\gamma}$ of the form
		\[y^2=x^3+a_2'x^2+a_4'x+a_6'.\]
		Similarly, since $\char(K)\neq 3$, we can then make the substitution $(x,y)\mapsto(x,y-\frac{1}{3}a_2')$ to obtain a Weierstrass equation in short form that is minimal at $\ord_{\gamma}$. It can be shown that if the initial equation is in short form, then the only change of variables such that the resulting equation is again in short form is of the form $(x,y)\mapsto(u^2x,u^3y)$ for some $u\in K^\times$. Since both our original equation and the equation obtained from the composition of these changes of variables are in short form, this composition of changes of variables is of the required form, thus proving the lemma. 
	\end{proof}
	\begin{lemma}\label{lem:heightvsdegree}
		Let $R$ be a non-zero point in $E(K)$, then $h(R)=2\deg(D_{R})+O_E(1)$.
	\end{lemma}
	\begin{proof}[Proof]
		There exist only finitely many points $\gamma\in \mathcal{C}(k)$ for which the Weierstrass equation of $E$ is not minimal at $\ord_{\gamma}$, say at all but $\gamma_1,\gamma_2,\ldots,\gamma_n\in \mathcal{C}(k)$ for some positive integer $n$. Using Lemma \ref{lem:descriptiondiv}, we have
		\begin{align*}h(R)=\deg(x(R))&=\sum_{\gamma\in \mathcal{C}(k)}\max\{0,-\ord_{\gamma}(x(R))\}=2\sum_{\gamma\in \mathcal{C}(k)}\max\{0,-1/2\ord_{\gamma}(x(R))\}\\& = 2\deg(D_R)+2\sum_{i=1}^n\left(\max\{0,-1/2\ord_{\gamma_i}(x(R))\}-\ord_{\gamma_i}(D_R)\right).\end{align*}
		\noindent By Lemmas \ref{lem:descriptiondiv} and \ref{lem:minimaleq}, there exists for each $1\leq i\leq n$ some $u_i\in K^\times$ such that $\ord_{\gamma_i}(D_R)=\max\{0,\ord_{\gamma_i}(u_i)-1/2\ord_{\gamma_i}(x(R))\}$, so above summands are bounded by $\ord_{\gamma_i}(u_i)$. Since $\ord_{\gamma_i}(u_i)$ depends only on $E$, this proves the lemma.
	\end{proof}
	We next work towards defining a \textit{canonical height} $\hat{h}:E(K)\to\mathbb{R}_{\geq 0}$ and stating some of its properties. We first require a preliminary proposition. Given points $R,S\in E(K)$, we let $(RS)$ denote the intersection number of the curves $(R):=\sigma_R(\mathcal{C})$ and $(S):=\sigma_S(\mathcal{C})$ on the surface $\Epsilon$. If $R$ is non-zero, then $(RO)$ is simply equal to $\deg D_R$.
	\begin{proposition}\label{prop:bilinform}
		Suppose $p>3.$ For all $R,S\in E(K)$, there exists a function $C(R,S,E)$ depending on $R,S,E$ such that the pairing $\langle\cdot,\cdot\rangle:E(K)\times E(K)\to \mathbb{R}$,
		\[\langle R,S\rangle=(RO)+(SO)-(RS)-C(R,S,E)\] is symmetric and bilinear. Moreover, the function $C(R,S,E)$ can be bounded independently of $R$ and $S$. If $R$ and $S$ are non-zero, then 
		\[\langle R,S\rangle=\deg D_R+\deg D_S-(RS)-C(R,S,E).\]
	\end{proposition}
	\begin{proof}[Proof]
		By \cite[Theorem 8.6]{shioda1990mordell}, this pairing is symmetric and bilinear, and $C(R,S,E)$ splits in $\chi(\Epsilon)+C'(R,S)$, where $\chi(\Epsilon)$ denotes the arithmetic genus of $\Epsilon$ and $C'(R,S)$ is some constant depending only on $R$ and $S$. Since $\Epsilon$ is constructed from $E$, the first assertion follows. In \cite[Lemma 7.3]{naskrkecki2016divisibility}, it is proved that the $C'(R,S)$ part can be bounded by another constant depending only on $\Epsilon$, so $C(R,S,E)$ can be bounded independently of $R$ and $S$. The last statement follows since $(RO)=\deg D_R$ if $R$ is non-zero and similarly for $S$.
	\end{proof}
	\begin{proposition}\label{prop:canonheight}
		Let $R,S\in E(K)$ be arbitrary points. There exists a map $\hat{h}:E(K)\to\mathbb{R}_{\geq 0}$ satisfying the following properties:
		\begin{enumerate}[noitemsep, label=(\roman*)]
			\item $\hat{h}(R)=\frac{1}{2}h(R)+O_E(1)$, and if $R$ is non-zero, then $\hat{h}(R)=\deg D_R+O_E(1)$;
			\item for all $j\in\mathbb{Z}$, $\hat{h}(jR)=j^2\hat{h}(R)$;
			\item $\hat{h}(R)=0$ if and only if $R$ is a torsion point.
		\end{enumerate}
		We call $\hat{h}$ the canonical height on $E(K)$. The pairing $\langle\cdot,\cdot\rangle:E(K)\times E(K)\to \mathbb{R}$ defined by $\langle R,S\rangle=\hat{h}(R+S)-\hat{h}(R)-\hat{h}(S)$ is bilinear. For $p>3$, this pairing coincides with the pairing in Proposition \ref{prop:bilinform}.
	\end{proposition}
	\begin{proof}[Proof]
		In the $p=0$ case, this is \cite[Theorem III.4.3]{silverman2013advanced} and Lemma \ref{lem:heightvsdegree}. Suppose $p>3$ and let $\langle\cdot,\cdot\rangle_1$ denote the pairing of Proposition \ref{prop:bilinform}. We define $\hat{h}(R)=\frac{1}{2}\langle R,R\rangle_1$ for all $R\in E(K)$. A direct computation shows that the pairings $\langle\cdot,\cdot\rangle$ and $\langle\cdot,\cdot\rangle_1$ are equal. Since $\langle\cdot,\cdot\rangle_1$ is bilinear, (ii) follows. Additionally, we obtain $\hat{h}(O)=0$, so for (i) we may assume that $R$ is non-zero. Then $(RO)=\deg D_R$, so by combining \cite[Lemma 2.7]{shioda1990mordell} with Proposition \ref{prop:bilinform}, we obtain that $\hat{h}(R)=\deg D_{R}+O_{E}(1)$ and (i) then follows from Lemma \ref{lem:heightvsdegree}. Assertion (iii) is \cite[Theorem 8.4]{shioda1990mordell}.
	\end{proof}
	We end this section with a lemma on height functions.
	\begin{lemma}\label{lem:heightbounds}
		Let $R,S$ be points in $E(K)$, then
		\begin{enumerate}[noitemsep, label=(\roman*)]
			\item there exists a positive constant $C_{R,S}$ that depends only on $R,S$ and $E$ such that $\hat{h}(nR+S)\geq \hat{h}(nR)-nC_{R,S,E};$
			\item $h(R+S)\leq 2h(R)+2h(S)+O_{E}(1).$
		\end{enumerate}
	\end{lemma}
	\begin{proof}[Proof]
		Let $\langle\cdot,\cdot\rangle$ denote the pairing of Proposition \ref{prop:canonheight}. We have
		\[0=\langle nR,S\rangle-n\langle R,S\rangle=\hat{h}(nR+S)-\hat{h}(nR)-\hat{h}(S)-n\left(\hat{h}(R+S)-\hat{h}(R)-\hat{h}(S)\right).\]
		Since $\hat{h}(T)\geq0$ for all $T\in E(K)$, we then obtain
		\[\hat{h}(nR+S)=\hat{h}(nR)+\hat{h}(S)+n\left(\hat{h}(R+S)-\hat{h}(R)-\hat{h}(S)\right)\geq \hat{h}(nR)-n\left(\hat{h}(R)+\hat{h}(S)\right).\]
		The first assertion then follows by putting $\hat{h}(R)+\hat{h}(S)=C_{R,S,E}$.
		
		For the second assertion, the statement is trivial if either $R$ or $S$ is zero, so assume that both are non-zero. We have in the $p=0$ case by \cite[Theorem III.4.2]{silverman2013advanced} that 
		\[h(R+S)=2h(R)+2h(S)-h(R-S)+O_E(1),\]
		and the result follows since $h(R-S)\geq 0$. If $p>3$, we have by Propositions \ref{prop:bilinform} and \ref{prop:canonheight} that
		\begin{align*}h(R+S)&=2\hat{h}(R+S)+O_E(1)\\&=2\left(\hat{h}(R)+\hat{h}(S)+\langle R,S\rangle\right)+O_E(1)\\&=h(R)+h(S)+2\left(\deg(D_R)+\deg(D_S)-(RS)-C(R,S,E)\right)+O_E(1).\end{align*}
		Since $(R)$ and $(S)$ are irreducible, it follows from \cite[Proposition 1.4]{hartshorne2013algebraic} that if $(R)\neq (S)$, then $(RS)\geq 0$. If $(R)=(S)$, we have by \cite[Lemma 2.7]{shioda1990mordell} that $(RS)=O_E(1)$. Since $C(R,S,E)$ can be bounded independent of $R,S$ (Proposition \ref{prop:bilinform}), it then follows by Lemma \ref{lem:heightvsdegree} that
		\begin{align*}
			h(R+S)\leq h(R)+h(S)+2\left(\deg(D_R)+\deg(D_S)\right)+O_E(1)=2h(R)+2h(S)+O_E(1),
		\end{align*}
		as desired.
	\end{proof}
	\subsection{Values of $D_R$ for specific points $R\in E(K)$}
	Suppose $\gamma\in\mathcal{C}(k)$ is a point. We let $K_\gamma$ denote the completion of $K$ at $\ord_{\gamma}$, and we let $R_\gamma$ denote the corresponding valuation ring with maximal ideal $\M_\gamma$. For $n$ a positive integer, we denote 
	\begin{align}\label{eq:defEKGamN}
		E(K)_{\gamma,n}:=\left\{R\in E(K)\setminus\{O\}:\ord_{\gamma}D_R\geq n\right\}\cup \{O\}.
	\end{align}
	Since $K$ can be embedded in $K_\gamma$, we can view $K$ as a subfield of $K_\gamma$. In particular, we can view $E$ as an elliptic curve over $K_\gamma$. We want to define a similar subset as (\ref{eq:defEKGamN}) for $E(K_\gamma)$, but we have not defined $D_R$ for general $R\in E(K_\gamma)$. We use Lemma \ref{lem:descriptiondiv} for this. For each $\gamma'\in \mathcal{C}(k)$, let $E_{\gamma'}$ be an elliptic curve that is minimal at $\ord_{\gamma'}$ and isomorphic to $E$ over $K$ with isomorphism $\phi_{\gamma'}:E\to E_{\gamma'}$.  One can show that $E_{\gamma'}$ is then also minimal at $\ord_{\gamma'}$ when considered as an elliptic curve over $K_\gamma$. For all non-zero $R\in E(K_\gamma)$, we define \[D_R:=\sum_{\gamma'\in \mathcal{C}(k)}\max\left\{0,-\frac{1}{2}\ord_{\gamma'}(x(\phi_{\gamma'}(R)))\right\}\gamma'.\]
	One can show that $D_R$ is an effective divisor on $\mathcal{C}$. Since $K$ is a subfield of $K_\gamma$, we can view $R\in E(K)$ as a point in $E(K_\gamma)$. By Lemma \ref{lem:descriptiondiv}, it follows that under this identification, above definition extends our earlier definition of $D_R$ for non-zero $R\in E(K)$. We define \[E(K_\gamma)_{\gamma,n}:=\left\{R\in E(K_\gamma)\setminus\{O\}:\ord_\gamma(D_R)\geq n\right\}\cup \{O\}.\] Under this identification, we then have $E(K)_{\gamma,n}\subset E(K_\gamma)_{\gamma,n}$. Using Lemma \ref{lem:descriptiondiv} and the formal group associated to an elliptic curve, one can show that $E(K)_{\gamma,n}$ and $E(K_\gamma)_{\gamma,n}$ are groups, and that one has a group isomorphism $E(K_\gamma)_{\gamma,n}\cong \M_\gamma^n$. See \cite[Chapter IV \& Proposition VII.2.2]{silverman2009arithmetic} for details.
	\begin{lemma}\label{lem:torsionptdiv0}
		Suppose $R$ is an $s$-torsion point in $E(K)$ for some integer $s>1$ that is not divisible by $p$. Then $D_R=0$.
	\end{lemma} 
	\begin{proof}[Proof]
		Suppose $\gamma\in \supp D_R$ and view $R$ as a point of $E(K_\gamma)$. Denote $d:=\ord_{\gamma}D_R>0$, then it follows from the discussion preceding this lemma that $R\in E(K_\gamma)_{\gamma,d}$. Let $[R]$ denote the image of $R$ in the quotient $E(K_\gamma)_{\gamma,d}/E(K_\gamma)_{\gamma,d+1}$, then $[R]$ is non-zero. By the discussion preceding this lemma, we have
		\[E(K_\gamma)_{\gamma,d}/E(K_\gamma)_{\gamma,d+1}\cong\M_\gamma^d/\M_\gamma^{d+1}\cong k.\]
		Since $p$ does not divide $s$, it then follows that $s[R]\neq O$, but this contradicts $s[R]=[sR]=[O]$. 
	\end{proof}
	Suppose $R$ is a non-torsion point of $E(K)$ and let $n$ be a positive integer. If $\gamma\in\supp D_R$, it is possible to relate the values $\ord_{\gamma}D_{nR}$ and $\ord_{\gamma}D_R$ through the formal group associated to an elliptic curve. This relation is much simpler in the $p=0$ case, so we will focus on the $p>0$ case. 
	
	Suppose $p>0$. We first require some notation. For each point $\gamma\in \mathcal{C}(k)$, let $E_\gamma$ denote an elliptic curve given by a Weierstrass equation that is minimal at $\ord_{\gamma}$ and isomorphic to $E$ over $K$. The following is from \cite[Chapter IV]{silverman2009arithmetic}. Fix some $\gamma\in \mathcal{C}(k)$, and let $\hat{E}_\gamma$ denote the formal group associated to $E_\gamma$. Then the multiplication-by-$p$ map $[p]:\hat{E}_\gamma\to\hat{E}_\gamma$ is defined by the formal power series $T\mapsto H_{E_\gamma}T^p+a_2T^{2p}+\ldots$. Since $E$ is ordinary, we have $H_{E_\gamma}\neq 0$, and since $E_\gamma$ is minimal at $\ord_{\gamma}$, we have $\ord_{\gamma}H_{E_\gamma}\geq 0$. The value $\ord_{\gamma}H_{E_\gamma}$ does not depend on the chosen $E_\gamma$. For each point $\gamma\in \mathcal{C}(k)$, we define \[h_{E,\gamma}:=\ord_{\gamma}H_{E_\gamma}.\] 
	We have $h_{E,\gamma}=\ord_{\gamma}H_E$ for $\gamma\in \mathcal{C}(k)$ outside the finite set of $\gamma'\in \mathcal{C}(k)$ for which $E$ is not minimal at $\ord_{\gamma'}$. Since $H_E$ has only finitely many zeroes, it follows that there are only finitely many points $\gamma\in \mathcal{C}(k)$ for which $h_{E,\gamma}\neq 0.$ 
	
	We next provide the proposition that relates $\ord_\gamma D_{nR}$ to $\ord_{\gamma}D_R$. The $p=0$ part is due to Ingram et al. \cite{ingram2012algebraic} and the $p>0$ part is due to Naskr{\k{e}}cki \cite{naskrkecki2016divisibility}. 
	\begin{proposition}[\hspace{1sp}{\cite[Lemma 5.6]{ingram2012algebraic} \& \cite[Lemma 8.2]{naskrkecki2016divisibility}}]\label{prop:valofmultiples}
		Suppose $R$ is a non-torsion point of $E(K)$ and $\gamma\in\supp D_{mR}$ for some positive integer $m$. Denote $m(\gamma):=\min\{n\geq1:\gamma\in\supp D_{nR} \}$ and let $n$ be a positive integer, then, 
		\begin{enumerate}[noitemsep, label=(\roman*)]
			\item if $m(\gamma)\nmid n$, $\ord_{\gamma}D_{nR}=0$;
			\item if $m(\gamma)\mid n$ and $p=0$, $\ord_{\gamma}D_{nR}=\ord_{\gamma}D_{m(\gamma)R}$;
			\item if $m(\gamma)\mid n$ and $p>0$, denote $e:=\ord_p\left(\frac{n}{m(\gamma)}\right)$. Then,
			\begin{enumerate}[noitemsep, label=(\alph*)]
				\item if $h_{E,\gamma}\leq p-1$, then $\ord_{\gamma}D_{nP}=p^e\ord_{\gamma}D_{m(\gamma)P}+\frac{p^e-1}{p-1}h_{E,\gamma}$;
				\item if $h_{E,\gamma}\geq p$, there exists an integer $j$, independent of $\gamma$ and depending only on $E$, and a function $\delta_{\gamma,m(\gamma)R}:\{0,1,\ldots,j\}\to \mathbb{Z}_{\geq 0}$, depending only on $\gamma,R$ and $E$, satisfying $\delta_{\gamma,m(\gamma)R}(0)=0$ and such that
				\[\ord_{\gamma}D_{nR}=\begin{cases}
					p^e\ord_{\gamma}D_{m(\gamma)R}+\delta_{\gamma,m(\gamma)R}(e), &\text{ if }e\leq j,\\
					p^e\ord_{\gamma}D_{m(\gamma)R}+\frac{p^{e-j}-1}{p-1}h_{E,\gamma}+p^{e-j}\delta_{\gamma,m(\gamma)R}(j), &\text{ otherwise}.
				\end{cases}\]
			\end{enumerate} 
		\end{enumerate}
	\end{proposition}
	%Suppose $p>0$ and let $n$ be a positive integer. Suppose $R\in E(K)$ is a non-torsion point and $\gamma\in\supp D_{nR}$. Let $m(\gamma)$ be the smallest positive integer for which $\gamma\in\supp D_{m(\gamma)R}$ and suppose $h_{E,\gamma}\geq p$. If we apply above proposition on both $m(\gamma)R$ and $nR$, we obtain two functions $\delta_{\gamma,m(\gamma)R}$ and $\delta_{\gamma, nR}$ that are related in the following way. Let $\ell$ be a positive integer, and denote $e=\ord_p(n/m(\gamma))$ and $s=\ord_p(\ell)$. Let $j$ be as in Proposition \ref{prop:valofmultiples} and assume $e,s\leq j$, then
	%\begin{align}\label{eq:relationDeltaGammas}
		%p^{s}\delta_{\gamma,m(\gamma)R}(e)+\delta_{\gamma,nR}(s)=\begin{cases}
		%\delta_{\gamma,m(\gamma)R}(e+s), &\text{ if }e+s\leq j,\\
		%\frac{p^{e+s-j}-1}{p-1}h_{E,\gamma}+p^{e+s-j}\delta_{\gamma,m(\gamma)R}(j), &\text{ otherwise.}
	%\end{cases}
	%\end{align}
	%\item if $e_1>j$, then $p^{e_2}\left(\frac{p^{e_1-j}-1}{p-1}h_{E,\gamma}+p^{e_1-j}\delta_{\gamma,R}(j)\right)+\delta_{\gamma, nR}(e_2)=\frac{p^{e_1+e_2-j}-1}{p-1}h_{E,\gamma}+p^{e_1+e_2-j}\delta_{\gamma,R}(j)$.	
	\begin{lemma}\label{lem:approxdelta}
		Suppose $p>0$. Suppose $R\in E(K)$ is a non-torsion point and let $n$ be a positive integer. Suppose $\gamma\in\supp D_{nR}$ with $h_{E,\gamma}\geq p$ and let $j$ be as in Proposition \ref{prop:valofmultiples}. Let $m(\gamma)$ be the smallest positive integer such that $\gamma\in\supp D_{m(\gamma)R}$ and denote $e=\ord_p(n/m(\gamma))$. Then for any non-negative integer $s\leq j$, we have \begin{align*}\delta_{\gamma, nR}(s)&=-p^{s}\delta_{\gamma, m(\gamma)R}(e)+\begin{cases}
				\delta_{\gamma,m(\gamma)R}(e+s), &\text{ if }e+s\leq j,\\
				\frac{p^{e+s-j}-1}{p-1}h_{E,\gamma}+p^{e+s-j}\delta_{\gamma,m(\gamma)R}(j), &\text{ otherwise,}
			\end{cases}\\&= O_{E,R}(n).\end{align*}
	\end{lemma}
	\begin{proof}[Proof]
		Fix some non-negative integer $s\leq j$. The first equality follows by applying Proposition \ref{prop:valofmultiples} on $np^sR$ for both $m(\gamma)R$ and $nR$ as initial point. There are only finitely many $\gamma'\in \mathcal{C}(k)$ for which $h_{E,\gamma'}\neq0$, so $C_1:=\max\{h_{E,\gamma'}:\gamma'\in \mathcal{C}(k)\}$ exists and depends only on $E$. Let $S$ denote the finite set of $\gamma'\in \mathcal{C}(k)$ for which $h_{E,\gamma'}\geq p$ and for which $\gamma'\in\supp D_{mR}$ for some positive integer $m$. Given $\gamma'\in S$, we let $m(\gamma')$ denote the smallest positive integer for which $\gamma'\in\supp D_{m(\gamma')R}.$ The constant $C_{2}:=\max\{\delta_{\gamma',m(\gamma')R}(t):0\leq t\leq j,\gamma'\in S\}$ then exists and only depends on $E$ and $R$. So the constant $C:=2\max\{C_1,C_{2}\}$ depends only on $E$ and $R$. Since $\delta_{\gamma, nR}(s)$ and $p^s\delta_{\gamma, m(\gamma)R}(e)$ are non-negative, it suffices for the second equality to show that $\delta_{\gamma, nR}(s)+p^s\delta_{\gamma, m(\gamma)R}(e)=O_{E,R}(n)$ and by the first equality we have $\delta_{\gamma, nR}(s)+p^s\delta_{\gamma, m(\gamma)R}(e)\leq p^{e}C\leq nC=O_{E,R}(n)$.
	\end{proof}
	\begin{corollary}\label{cor:ApproxValDnr}
		Suppose $R\in E(K)$ is a non-torsion point and let $\gamma\in C(k)$. For each positive integer $n$, we have $\ord_{\gamma}D_{nR}=O_{E,R,\gamma}(n).$ 
	\end{corollary}
	\begin{proof}[Proof]
		We may assume that $\ord_{\gamma}D_{nR}>0$. Let $m(\gamma)$ be the smallest positive integer such that $\gamma\in\supp D_{m(\gamma)R}$. Denote $e=\ord_p(n/m(\gamma))$ if $p>0$ and $e=0$ if $p=0$. By Proposition \ref{prop:valofmultiples} and (the proof of) Lemma \ref{lem:approxdelta}, we obtain that
		\[\ord_{\gamma}D_{nR}- p^e\ord_{\gamma} D_{m(\gamma)R}=p^eO_{E,R}(m(\gamma))\leq nO_{E,R}(m(\gamma))=O_{E,R,\gamma}(n).\]
		Since $\ord_{\gamma}D_{m(\gamma)R}$ depends only on $E,R$ and $\gamma$, we have $p^e\ord_{\gamma} D_{m(\gamma)R}=O_{E,R,\gamma}(n)$, from which the result then follows.
	\end{proof}
	\section{Proof of Theorem \ref{thm:main}}\label{sec:prfmain}
	If $r=1$, then the proof is due to Ingram et al. if $p=0$ \cite[Theorem 1.7]{ingram2012algebraic} and due to Naskr{\k{e}}cki if $p>3$ \cite[Theorem 8.11]{naskrkecki2016divisibility}. So we may assume $r>1$.  We denote $S:=\bigcup_{b\mid r, b<r}\supp D_{bQ},$ then $S$ is finite. Moreover, if $p$ does not divide $r$, then $S$ is empty by Lemma \ref{lem:torsionptdiv0}. The next proposition and corollary are key ingredients of the proof. These are direct extensions of the analogous results of Verzobio in the number field case, see \cite{verzobio2020primitive}.
	\begin{proposition}\label{prop:gammainsupprndaswell}
		Let $n$ be a positive integer and suppose $D_{nP+Q}$ does not have a primitive divisor. Suppose $\gamma$ lies in the support of $D_{nP+Q}$ and does not lie in $S$. Then there exists a positive integer $d>r$ that divides $n$ and is coprime with $r$ such that $\gamma$ lies in the support of $D_{\frac{rn}{d}P}$ as well.
	\end{proposition} 
	\begin{proof}[Proof]
		Since $D_{nP+Q}$ does not have a primitive divisor, there exists an integer $1\leq j< n$ such that $\gamma\in\supp D_{(n-j)P+Q}$. So both $nP+Q$ and $(n-j)P+Q$ are elements of $E(K)_{\gamma,1}$ and since $E(K)_{\gamma,1}$ is a group, we then have $jP\in E(K)_{\gamma,1}$. Similarly, we have that $r(nP+Q)=rnP\in E(K)_{\gamma,1}$, so for $s=\gcd(rn,j)$, we have $sP\in E(K)_{\gamma,1}$. Write $s=\frac{rn}{d}$ for some positive integer $d$ and denote $c=\gcd(r,d)$. Now write $r=r_1c$ and $d=d_1c$, then $s=\frac{r_1n}{d_1}$. Now if $c>1$, then $r_1Q=r_1(nP+Q)-d_1sP\in E(K)_{\gamma,1}$, which contradicts that $\gamma\notin S$ since $r_1\mid r$ and $r_1<r$. So $c=1$ and $d$ is coprime with $r$. Since $d$ divides $rn$, it then follows that $d$ divides $n$. Since $s$ divides $j$ and $j<n$, we have $d>r$. Since $\gamma\in\supp D_{sP}$ and $s=\frac{rn}{d}$, the proposition is proved.
	\end{proof}
	To improve readability, we write $h_\gamma(R):=\ord_{\gamma}D_R$ for a non-zero point $R\in E(K)$ and $\gamma\in \mathcal{C}(k)$.
	\begin{corollary}\label{cor:boundbyassociatedd}
		Assume the same hypotheses as in the preceding proposition. Let $d$ be the positive integer obtained from that proposition. If $p=0$, put $e=0$ and if $p>0$, put $e=\ord_p(d)$. There then exist non-negative integers $b<r$ and $\epsilon_{d,\gamma}$, where $b$ depends only on $d$ and $r$, such that \[h_\gamma(nP+Q)\leq p^eh_\gamma\left(\frac{n}{d}P+bQ\right)+\epsilon_{d,\gamma}.\]
		Moreover, let $j$ and $\delta_{\gamma,\frac{rn}{d}P}$ be as in Proposition \ref{prop:valofmultiples}(iii), then
		\[\epsilon_{d,\gamma}=\begin{cases}
			\frac{p^e-1}{p-1}h_{E,\gamma},\ &\text{ if }h_{E,\gamma} <p,\\
			\delta_{\gamma,\frac{rn}{d}P}(e), &\text{ if }h_{E,\gamma}\geq p \text{ and }e\leq j,\\
			\frac{p^{e-j}-1}{p-1}h_{E,\gamma}+p^{e-j}\delta_{\gamma,\frac{rn}{d}P}(j), &\text{ otherwise}.
		\end{cases}\]
	\end{corollary}
	\begin{proof}[Proof]
		We denote $P_1=\frac{rn}{d}P$ and $P_2=nP+Q$. Since $\gcd(r,d)=1=\gcd(r,d-r)=1$, there exists $a,c\in\mathbb{Z}$ such that $ar+c(d-r)=1$ and so
		\[\frac{n}{d}P+cQ=ar\frac{n}{d}P+c(d-r)\frac{n}{d}P+cQ=(a-c)\frac{rn}{d}P+c(nP+Q)=(a-c)P_1+cP_2.\]
		First suppose $p\mid r$, then $p\nmid d$. Since $h_\gamma(P_1),h_\gamma(P_2)\geq 1$, we then have by Proposition \ref{prop:valofmultiples} that $h_{\gamma}(P_1)=h_{\gamma}(dP_1)=h_{\gamma}(rP_2)\geq h_{\gamma}(P_2).$ Denote $s:=h_{\gamma}(P_2)$, then $P_1,P_2\in E(K)_{\gamma,s}$ and so $\frac{n}{d}P+cQ=(a-c)P_1+cP_2\in E(K)_{\gamma,s}$. Since $\frac{n}{d}P+cQ$ is non-zero, we then have $h_{\gamma}(nP+Q)=h_{\gamma}(P_2)=s\leq h_{\gamma}(\frac{n}{d}P+cQ).$
		
		Now suppose $p\nmid r$. Again, by Proposition \ref{prop:valofmultiples}, we then have $h_\gamma(P_2)=h_\gamma(rP_2)=h_\gamma(dP_1)=p^eh_\gamma(P_1)+\epsilon_{d,\gamma}.$	Denote $t:=h_\gamma(P_1)\geq 1$, then $P_1,P_2\in E(K)_{\gamma,t}$ and so $\frac{n}{d}P+cQ\in E(K)_{\gamma,t}$. We obtain $h_\gamma(\frac{n}{d}P+cQ)\geq t$ and so $h_\gamma(nP+Q)=h_\gamma(P_2)=p^eh_\gamma(P_1)+\epsilon_{d,\gamma}\leq p^eh_\gamma\left(\frac{n}{d}P+cQ\right)+\epsilon_{d,\gamma}$. In both cases, the corollary follows by using that $Q$ is an $r$-torsion point and putting $0\leq b<r$ with $b\equiv c\pmod{r}$.
	\end{proof}
	We are now able to prove Theorem \ref{thm:main}. Suppose $n$ is a positive integer such that $D_{nP+Q}$ does not have a primitive divisor. Combining Proposition \ref{prop:canonheight} with Lemma \ref{lem:heightbounds}, we have for some positive constant $C_{P,Q,E}$, depending only on $P,Q$ and $E$, that
	\begin{align}\label{eq:approxCanonHeightAsDegDNPQ}n^2\hat{h}(P)=\hat{h}(nP)\leq \hat{h}(nP+Q)+nC_{P,Q,E}&=\deg D_{nP+Q}+O_{E,P,Q}(n) \nonumber\\
		&=\sum_{\gamma\in \supp D_{nP+Q}}h_\gamma(nP+Q)+O_{E,P,Q}(n).\end{align}
	We will apply Corollary \ref{cor:boundbyassociatedd} to bound the latter sum. However, we can not apply Corollary \ref{cor:boundbyassociatedd} to the $\gamma\in\supp D_{nP+Q}$ that also lie in $S$. For those, we use the next lemma.
	\begin{lemma}\label{lem:ApproxPointsS}
		$\sum_{\gamma\in S}h_\gamma(nP+Q)=O_{E,P,Q}(n).$
	\end{lemma}
	\begin{proof}[Proof]
		By Proposition \ref{prop:valofmultiples}, we have for each $\gamma\in C(k)$ that $h_{\gamma}(nP+Q)\leq h_{\gamma}(r(nP+Q))=h_{\gamma}(rnP)$, since $Q$ has order $r$. By Corollary \ref{cor:ApproxValDnr}, we have $h_{\gamma}(rnP)=O_{E,P,\gamma}(rn)$. Combining, we obtain 
		\[\sum_{\gamma\in S}h_\gamma(nP+Q)\leq \sum_{\gamma\in S}h_\gamma(rnP)=\sum_{\gamma\in S}O_{E,P,\gamma}(rn)=O_{E,P,Q}(n),\]
		where the last step follows since both $r$ and $S$ depend only on $E$ and $Q$.
	\end{proof}
	Denote $T:=\supp D_{nP+Q}\setminus S$. By Proposition \ref{prop:gammainsupprndaswell}, we find for each $\gamma\in T$ an associated positive integer $d_\gamma$ dividing $n$, coprime with $r$, and larger than $r$. We define \[\mathcal{D}_n:=\{d\in\mathbb{N}:d\mid n,\ d>r\ \text{and}\ \gcd(d,r)=1\}.\] Given $d\in \mathcal{D}_n$, we obtain from the proof of Corollary \ref{cor:boundbyassociatedd} an associated non-negative integer $b_d<r$. Given a positive integer $d$, we denote $e_{d}=0$ if $p=0$ and $e_d=\ord_p(d)$ if $p>0$. Suppose $\gamma\in T$, then we have by Proposition \ref{prop:gammainsupprndaswell} and Corollary \ref{cor:boundbyassociatedd} that
	\[h_\gamma(nP+Q)\leq p^{e_{d_\gamma}}h_\gamma\left(\frac{n}{d_\gamma}P+b_{d_\gamma} Q\right)+\epsilon_{d_\gamma,\gamma}.\] 
	Since $b_{d_\gamma}$ only depends on $d_\gamma$ and $r$, we obtain for any divisor $d\in \mathcal{D}_n$ an associated non-negative integer $b_d<r$, such that above inequality holds if $d=d_\gamma$ for some $\gamma\in T$. We can thus make the approximation
	\begin{align}\label{eq:approxSumTpart}\sum_{\gamma\in T}h_\gamma(nP+Q)&\leq\sum_{\gamma\in T} p^{e_{d_\gamma}}h_\gamma\left(\frac{n}{d_\gamma}P+b_{d_\gamma} Q\right)+\epsilon_{d_\gamma,\gamma} \nonumber\\
		&\leq\sum_{\substack{d\in \mathcal{D}_n}}\ \sum_{\gamma\in T} p^{e_{d}}h_\gamma\left(\frac{n}{d}P+b_{d} Q\right)+\underbrace{\sum_{\gamma\in T}\epsilon_{d_\gamma,\gamma}}_{W(n,P,Q)} \nonumber\\
		&\leq\sum_{\substack{d\in \mathcal{D}_n}}p^{e_d}\deg\left(D_{\frac{n}{d}P+b_dQ}\right)+W(n,P,Q) \nonumber\\
		&=\sum_{\substack{d\in \mathcal{D}_n}}p^{e_d}\left(\frac{1}{2}h\left(\frac{n}{d}P+b_dQ\right)+O_E(1)\right)+W(n,P,Q),
	\end{align}
	where the last equality follows from Lemma \ref{lem:heightvsdegree}.
	\begin{lemma}\label{lem:approxwnpq}
		$W(n,P,Q)=O_{E,P,Q}(n)$.
	\end{lemma}
	\begin{proof}[Proof]
		Fix some $\gamma\in T$ for which $\epsilon_{d_{\gamma},\gamma}>0$. Since $r<d_\gamma\leq n$, it follows from (the proof of) Lemma \ref{lem:approxdelta} and the definition of $\epsilon_{d_\gamma,\gamma}$ that $\epsilon_{d_\gamma,\gamma}=O_{E,P}(rn)=O_{E,P,Q}(n)$. If $\gamma'\in T$ such that $h_{E,\gamma'}=0$, then $\epsilon_{d_{\gamma'},\gamma'}=0$. The lemma then follows since there are only finitely many $\gamma\in \mathcal{C}(k)$ such that $h_{E,\gamma}\neq 0$ and $\#\{\gamma\in \mathcal{C}(k):h_{E,\gamma}\neq 0\}$ depends only on $E$. 
	\end{proof}
	By Lemma \ref{lem:heightbounds}, we find a constant $C:=\max_{0\leq b<r}C_{bQ}$, depending only on $Q$ and $E$, such that for each $d\in \mathcal{D}_n$, we have $h\left(\frac{n}{d}P+b_dQ\right)\leq 2h\left(\frac{n}{d}P\right)+C.$ Combining  this with Proposition \ref{prop:canonheight} and the definition of $\mathcal{D}_n$, we obtain
	\begin{align}\label{eq:convertHeightToCanonical}
		\sum_{\substack{d\in \mathcal{D}_n}}p^{e_d}\left(\frac{1}{2}h\left(\frac{n}{d}P+b_dQ\right)+O_E(1)\right)&\leq \sum_{\substack{d\in \mathcal{D}_n}}p^{e_d}\left(h\left(\frac{n}{d}P\right)+O_{E,Q}(1)\right) \nonumber\\
		&= \sum_{\substack{d\in \mathcal{D}_n}}p^{e_d}\left(2\hat{h}\left(\frac{n}{d}P\right)+O_{E,Q}(1)\right) \nonumber\\
		&\leq 2n^2\hat{h}\left(P\right)\sum_{\substack{d\in \mathcal{D}_n}}p^{e_d}\frac{1}{d^2}+ \sum_{d\mid n}p^{e_d}O_{E,Q}(1).
	\end{align}
	For the last term, we apply the following lemma.
	\begin{lemma}
		For any positive constant $\alpha\in\mathbb{R}$, we have $\sum_{\substack{d\mid n}}p^{e_d}=o(n^{1+\alpha})$.
	\end{lemma}
	\begin{proof}[Proof]
		The statement is immediate if $p=0$, so suppose $p>0$ and denote $e:=\ord_p(n)$. We let $\delta:\mathbb{N}\to\mathbb{N},m\mapsto\sum_{d\mid m}1$ denote the divisor function, then one can show that \[\sum_{d\mid n}p^{\ord_p(d)}=\frac{p^{e+1}-1}{(e+1)(p-1)}\delta(n)\leq 2p^e\delta(n)\leq 2n\delta(n).\]
		By \cite[p.296]{apostol2013introduction}, we have for any positive constant $\alpha$ that $\delta(n)=o(n^\alpha),$ so the lemma follows.
	\end{proof}
	For the rest of this section, fix some constant $0<\alpha<1$, then $\sum_{\substack{d\mid n}}p^{e_d}O_{E,Q}(1)=o(n^{1+\alpha})$. Combining this with Lemma \ref{lem:approxwnpq}, (\ref{eq:approxSumTpart}) and (\ref{eq:convertHeightToCanonical}), we have proved the following corollary.
	\begin{corollary}\label{cor:approxSumT}
		$\sum_{\gamma\in T}h_\gamma(nP+Q)\leq 2n^2\hat{h}\left(P\right)\sum_{d\in \mathcal{D}_n}p^{e_d}\frac{1}{d^2}+o(n^{1+\alpha})$.
	\end{corollary} 
	 Putting everything together, it follows by (\ref{eq:approxCanonHeightAsDegDNPQ}), Lemma \ref{lem:ApproxPointsS} and Corollary \ref{cor:approxSumT} that if $D_{nP+Q}$ does not have a primitive divisor, then
	\begin{align*}
		n^2\hat{h}(P)=\sum_{\gamma\in \supp D_{nP+Q}}h_\gamma(nP+Q)+O_{E,P,Q}(n)&=\sum_{\gamma\in T}h_{\gamma}(nP+Q)+\sum_{\gamma\in S}h_\gamma(nP+Q)+O_{E,P,Q}(n)\\
		&\leq2n^2\hat{h}\left(P\right)\sum_{d\in \mathcal{D}_n}p^{e_d}\frac{1}{d^2}+o(n^{1+\alpha}).\end{align*}
	Since $\hat{h}(P)>0$ by Proposition \ref{prop:canonheight} and $0<\alpha<1$, we see that if $\sum_{d\in \mathcal{D}_n}p^{e_d}\frac{1}{d^2}<1/2$ for all $n$, then above inequality can only hold for bounded $n$. So the theorem follows if we can prove that when $p$ and $r$ are entries in Table \ref{tbl:mainthm}, then $\sum_{d\in \mathcal{D}_n}p^{e_d}\frac{1}{d^2}<1/2$ for all $n$. If $p=0$ or $p>3$ and $p\mid r$, then $e_d=0$ for all $d\in \mathcal{D}_n$, so 
	\begin{align}\label{eq:approxZetapeq0}
		\sum_{d\in \mathcal{D}_n}p^{e_d}\frac{1}{d^2}=\sum_{\substack{d\in \mathcal{D}_n}}\frac{1}{d^2}\leq \sum_{\substack{d\mid n,d>2}}\frac{1}{d^2}\leq  \zeta(2)-1-\frac{1}{4}\approx .395<1/2.
	\end{align}
	Next, assume $p>3$. We are left with the values in Table \ref{tbl:mainthm} with $p\nmid r$. Denote $e=\ord_{p}(n)$ and write $n=n_0p^e$. Then
	\[\sum_{d\in \mathcal{D}_n}p^{e_d}\frac{1}{d^2}\leq\sum_{\substack{d\mid n,d>r}}p^{\ord_p(d)}\frac{1}{d^2}=\sum_{d_0\mid n_0,d_0>r}\frac{1}{d_0^2}+\sum_{i=1}^ep^{-i}\sum_{\substack{d_0\mid n_0,d_0p^i>r}}\frac{1}{d_0^2}.\]
	Using that $\sum_{i=1}^e p^{-i}=\frac{1-p^{-e}}{p-1}< \frac{1}{p-1}$, it follows that
	\begin{align*}\sum_{d_0\mid n_0,d_0>r}\frac{1}{d_0^2}+\sum_{i=1}^ep^{-i}\sum_{\substack{d_0\mid n_0,d_0p^i>r}}\frac{1}{d_0^2}&\leq \zeta(2)-1-1/4-1/9-\ldots-1/r^2+\frac{1}{p-1}\sum_{d_0\mid n_0}\frac{1}{d_0^2}\\&\leq \zeta(2)\left(1+\frac{1}{p-1}\right)-1-1/4-1/9-\ldots-1/r^2.\end{align*}
	A calculation shows that the entries in Table \ref{tbl:mainthm} with $p$ not dividing $r$ are precisely those for which 
	\[\zeta(2)\left(1+\frac{1}{p-1}\right)-1-1/4-1/9-\ldots-1/r^2<1/2,\]
	thus finishing the proof of Theorem \ref{thm:main}.
	\section{Necessity of the conditions in Theorem \ref{thm:main}}\label{sec:conditions}
	We end this paper by discussing the necessity of some of the hypotheses in Theorem \ref{thm:main}. 
	\subsection{Assumption that the elliptic curve is ordinary}
	Suppose $p>0$. Suppose all our previous assumptions hold, except that $E$ is no longer ordinary, and we also allow $p=2,3$. First consider the sequence $\{D_{nP}\}$. In \cite[Section 9]{naskrkecki2016divisibility}, it is shown that there then exist examples for which there does not exist a bound $N$ such that $D_{nP}$ has a primitive divisor for all $n\geq N$. We extend the constructions in [loc. cit.] to obtain counterexamples for the sequence $\{D_{nP+Q}\}$ as well.
	\begin{example}\label{ex:counter}
		Suppose $p>2$ is a prime number. Let $\alpha,\beta\in\mathbb{F}_p$ be such that $E_0:y^2=x^3+\alpha x+\beta$ is a supersingular elliptic curve (this is possible by \cite[Theorem 14.18]{cox2011primes} for $p\geq 5$ and for $p=3$ we take the equation $y^2=x^3+x$). Following \cite[Example 9.3]{naskrkecki2016divisibility}, consider the function field $K_0:=\mathbb{F}_p(t)$ and put $s=t^3+\alpha t+\beta$. The curve $E_0$ is then isomorphic over the algebraic closure $\overline{K_0}$ to the elliptic curve $E/K_0$ given by the equation $y^2=x^3+\alpha s^2x+\beta s^3$ through the isomorphism $(x,y)\mapsto (xs,ys^{3/2})$. Since $E_0$ is defined over $\mathbb{F}_p$, we have for $[p]:E_0\to E_0$ that $x\left([p](x,y)\right)=x^{p^2}$, see \cite[Exercise 5.16]{silverman2009arithmetic}. Combining with the isomorphism $E_0\cong E$, we have for each positive integer $k$ and $p^k:E\to E$ that $x\left([p^k](x,y)\right)=s\left(\frac{x}{s}\right)^{p^{2k}}$. From this formula, we deduce that $P=(ts, s^2)\in E(K_0)$ is non-torsion. Denote $K:=\overline{\mathbb{F}_p}(t)$ and let $L/K$ be some finite field extension that does not contain $s^{1/2}$. Then $E$ is not $L$-isomorphic to an elliptic curve defined over $\overline{\mathbb{F}_p}$ and $L$ is a function field over an algebraically closed field. Suppose $Q\in E(L)$ is an $r$-torsion point for some integer $r>1$. Since $E_0$ is supersingular, $p$ does not divide $r$ and so there exists a positive integer $k$ such that $p^k\equiv 1\pmod{r}$. Fix such an integer $k$ and denote $P+Q=(x',y')$. Then, for each positive integer $\ell$, we have $[p^{\ell k}](P+Q)=p^{\ell k}P+Q$ and so $x\left([p^{\ell k}]P+Q\right)= s\left(\frac{x'}{s}\right)^{p^{2\ell k}}$. From this expression, it follows that there are infinitely many terms in the sequence of divisors $\{D_{nP+Q}\}$ that do not have a primitive divisor. To finish the counterexample, we are left with proving the existence of a field extension $L/K$ such that $E(L)[r]$ is non-trivial for some positive integer $r>1$ and $s^{1/2}\notin L$. Consider the $2$-torsion on $E$. The non-zero $2$-torsion points in $E(\overline{K})$ are given by the points $(\gamma,0)$ with $\gamma$ a root of $f:=X^3+\alpha s^2 X+\beta s^3\in K[X]$. Since $f$ is a degree $3$ polynomial over $K$, either $f$ contains a root in $K$, or $f$ is irreducible over $K$. In the first case, $E(K)$ already contains a non-trivial $2$-torsion point and we can take $L=K$. In the second case, we let $L$ be the field obtained by adjoining a root of $f$ to $K$, and it follows by comparing degrees that $s^{1/2}\notin L$. In both cases, $E(L)[2]$ is not trivial and $s^{1/2}\notin L$, so this produces a counterexample.
		
		A similar approach works for $p=2$. The elliptic curve $E_0:y^2+y=x^3$ is supersingular over $\mathbb{F}_2$. Denote $K_0=\mathbb{F}_2(t)$, then the curve $E/K_0$ given by the equation $y^2+(t^3-1)y=x^3$ is isomorphic to $E_0$ over an algebraic closure $\overline{K_0}$ of $K_0$. Namely, fix some root $\alpha\in\overline{K_0}$ to the equation $f:=X^3-t^3+1$ in $K_0[X]$, then an isomorphism $E\to E_0$ is given by $(x,y)\mapsto(\alpha^{-2}x,\alpha^{-3}y)$. A calculation shows that $f$ is irreducible over $K:=\overline{\mathbb{F}_p}(t)$, so $E$ is not $K$-isomorphic to an elliptic curve defined over $\overline{\mathbb{F}_p}$. On $E_0$, we have $x\left([2](x,y)\right)=x^4$, so on $E$ we have for each positive integer $k$ that $x([2^k](x,y))=x^{4^k}\alpha^{2(1-4^k)}$. We deduce that the point $P=(t,1)\in E(K)$ is non-torsion and it again follows, similar to the $p>2$ case, that for $L/\overline{\mathbb{F}_2}(t)$ some finite field extension and $Q\in E(L)$ a torsion point of order $r>1$, there are infinitely many terms in the sequence of divisors $\{D_{nP+Q}\}$ that do not have a primitive divisor. The point $(0,0)\in E(K)$ is $3$-torsion, so since $f$ is irreducible over $K$, we can take $L=K$ and $Q=(0,0)$ to produce a counterexample. 
	\end{example}
	\begin{remark}
		The point $Q$ in Example \ref{ex:counter} has small order. Let us explain why this is necessary in our counterexample. Let $p$ be a prime number and use the same notation as in Example \ref{ex:counter}. Let $Q'\in E(\overline{K})$ be some torsion point of order $\ell>1$, then we showed that the sequence $\{D_{nP+Q'}\}$ will contain infinitely many terms that do not have a primitive divisor. However, the issue is that if $L/K$ is a field extension such that $Q'\in E(L)$, then $E$ will be isomorphic to $E_0$ over $L$ unless $\ell=2$ if $p>2$ and $\ell=3$ if $p=2$. To see this, let $\phi:E\to E_0$ denote the isomorphism, then we have the description $E[\ell]=\phi^{-1}(E_0[\ell])$. Suppose $p>2$, then $\phi^{-1}$ maps $(x,y)$ to $(xs^{-1},ys^{-3/2})$. So $Q'=(xs^{-1},ys^{-3/2})$ for certain $x,y\in\overline{\mathbb{F}_p}$. It follows that if $Q'\in E(L)$ for some field extension $L/K$, then $s^{1/2}\in L$ unless $y=0$, which is the case if and only if $\ell=2$. The $p=2$ case works similarly.
	\end{remark}
	
	\subsection{Assumption that $k$ is algebraically closed}
	It is possible to relax the condition of $k$ being algebraically closed if we assume that our elliptic curve is not isomorphic over $\overline{k}K$ to an elliptic curve over $\overline{k}$. We fix the following notation for this subsection. Let $k$ be a field of characteristic $p$ and let $\mathcal{C}/k$ be a non-singular, projective and geometrically integral curve. Let $\overline{k}$ denote an algebraic closure of $k$ and let $\mathcal{C}_{\overline{k}}$ denote the base extension of $\mathcal{C}$ to $\overline{k}$. Let $K$ denote the function field of $\mathcal{C}$ and let $K'=\overline{k}K$ denote the function field of $\mathcal{C}_{\overline{k}}$. Let $E/K$ be an ordinary elliptic curve. Suppose $P\in E(K)$ is a non-torsion point and suppose $Q\in E(K)$ is a torsion point of order $r$. We let $E_{K'}$ denote the base extension of $E$ to $K'$.
	
	Given a non-zero point $R\in E(K)$, we can define an effective divisor $D'_R\in \Div(\mathcal{C})$ similar to what we did in the algebraically closed case. Let $|\mathcal{C}|\subset \mathcal{C} $ denote the subset of closed points. Given $\gamma\in |\mathcal{C}|$, we have a corresponding valuation $v_\gamma$ on $K$, and an elliptic curve $E_\gamma/K$ that is minimal at $v_\gamma$ and isomorphic to $E$ over $K$. Let $\phi_\gamma:E\to E_\gamma$ denote this isomorphism, then we obtain for each $\gamma\in |\mathcal{C}|$ a non-negative integer $n_{\gamma,R}:=\max\{0,-1/2v_\gamma(x(\phi_\gamma(R)))\}$. There will only be finitely many $\gamma\in |\mathcal{C}|$ such that $n_{\gamma,R}\neq 0$. We define the effective divisor $D'_R:=\sum_{\gamma\in |\mathcal{C}|}n_{\gamma,R}\gamma\in\Div(\mathcal{C})$. Given a sequence of non-zero points $\{P_n\}\subset E(K)$, we again say that $P_n$ has a \textit{primitive divisor} if there exists $\gamma\in\supp D'_{P_n}$ such that $\gamma\notin\supp D'_{P_m}$ for any $1\leq m<n$.
	\begin{corollary}
		Suppose $E$ is not isomorphic over $K'$ to some elliptic curve $E_0/\overline{k}$. If either $r=1$ and $p\neq 2,3$ or the values of $p$ and $r$ are entries in Table \ref{tbl:mainthm}, then $D'_{nP+Q}$ has a primitive divisor for all $n$ sufficiently large.
	\end{corollary}
	\begin{proof}[Proof]
		Under the hypotheses, we know by Theorem $\ref{thm:main}$ that there exists a bound $N_1$ such that $D_{nP+Q}$ has a primitive divisor for all $n\geq N_1$. Let $\gamma_1,\ldots,\gamma_m$ be the points in $\mathcal{C}_{\overline{k}}(\overline{k})$ such that $E_{K'}$ is not minimal at $\ord_{\gamma_i}$. Let $N_2$ be a positive integer such that if $\gamma_i\in \supp D_{nP+Q}$ for some positive integer $n\geq N_2$, then $\gamma_i\in\supp D_{mP+Q}$ as well for some $1\leq m <N_2$. Put $N=\max\{N_1,N_2\}$ and let $n\geq N$, then there exists $\gamma\in\supp D_{nP+Q}$ with $\gamma\notin\supp D_{mP+Q}$ for any $1\leq m<n$. Let $\gamma'\in |\mathcal{C}|$ be such that $\ord_{\gamma}\vert_K=v_{\gamma'}$. Then $E$ is minimal at $v_{\gamma'}$ since $E_{K'}$ is minimal at $\ord_{\gamma}$. Since $x(nP+Q)\in K$, we have $-\frac{1}{2}v_{\gamma'}(x(nP+Q))=-\frac{1}{2}\ord_{\gamma}(x(nP+Q))>0$ and so $\gamma'\in \supp D'_{nP+Q}$. Similarly, if $\gamma'\in \supp D'_{mP+Q}$ for some $1\leq m <n$, then $\gamma\in \supp D_{mP+Q}$, which we assumed not to be the case. So $D'_{nP+Q}$ has a primitive divisor for all $n\geq N$.
	\end{proof}
	
	\subsection{Remaining pairs of $p$ and $r$}
	The $p\neq 2,3$ assumption is used at several steps in our proof. Most importantly, if $p=2$ or $p=3$ and $r>1$ any integer, then
	\begin{align}\label{eq:zeta2}\zeta(2)\left(1+\frac{1}{p-1}\right)-1-1/4-\ldots-1/r^2\geq \zeta(2)\left(1+\frac{1}{2}\right)-\zeta(2)=\frac{1}{2}\zeta(2)\approx 0.822> 1/2.
	\end{align}
	For a reason similar to the equality not holding in (\ref{eq:zeta2}), the proof that $D_{nP}$ has a primitive divisor for all $n$ sufficiently large does not work if $p=2,3$. To the author's knowledge, this is still an open problem. Interestingly enough, we see from our proof that if $p=2,3$ and $p$ divides the order of $Q$, then (\ref{eq:approxZetapeq0}) does hold, however our proof does not work in this case because we already use the assumption that $p\neq 2,3$ in Section \ref{sec:prelims}. All in all, it would be very interesting to further investigate the remaining $p=2,3$ case, either for $D_{nP+Q}$ or the classical $D_{nP}$ case. Additionally, it would be interesting to investigate what happens for the remaining pairs of $p$ and $r$, or what happens if $Q$ is a non-torsion point.
	\subsection*{Acknowledgements}
	This work originated from the author's master's thesis, which was written under supervision of Gunther Cornelissen. The author would like to thank him for the valuable conversations during this time. The author would also like to express his gratitude towards Jeroen Sijsling and Matteo Verzobio for useful comments on an earlier draft of this paper. 
	
	{\bibliography{PrimDivECFFArbitraryOrder}}
	\bibliographystyle{plain}
	
\end{document}